\documentclass[12pt,reno]{amsart}
\usepackage{amssymb}
\usepackage{amsmath}
\usepackage{enumitem}
\usepackage{mathrsfs}
\usepackage[all]{xy}
\usepackage{hyperref}  
\usepackage{marginnote}
\usepackage[english]{babel}

\usepackage[margin=1.5in]{geometry}

\RequirePackage{mathrsfs} 
\usepackage[OT2,T1]{fontenc}
\DeclareSymbolFont{cyrletters}{OT2}{wncyr}{m}{n}
\DeclareMathSymbol{\Sha}{\mathalpha}{cyrletters}{"58}

\newtheorem{thm}{Theorem}[section]
\newtheorem{lem}[thm]{Lemma}

\newtheorem{cor}[thm]{Corollary}

\theoremstyle{definition}

\newtheorem{opm}[thm]{Remark}
\newtheorem{defn}[thm]{Definition}

\numberwithin{equation}{section} \numberwithin{figure}{section}

\DeclareMathOperator{\prim}{prim}
\DeclareMathOperator{\Hom}{Hom}

\newcommand{\Qbar}{\overline{\QQ}}

\newcommand{\isomto}{\overset{\sim}{\longrightarrow}}

\newcommand\PP{\mathbb{P}}

\newcommand\QQ{\mathbb{Q}}

\newcommand\CC{\mathbb{C}}

\title[Belyi's theorem for complete intersections]{Belyi's theorem for   complete intersections of general type}

\author{A. Javanpeykar}


\address{Mathematical Institute \\ Leiden University\\
Leiden, Netherlands}

\email{ajavanp@math.leidenuniv.nl}

\begin{document}

\thispagestyle{empty}

\begin{abstract}
We give a Belyi-type characterisation of smooth complete intersections of general type over $\CC$ which can be defined over $\Qbar$.   Our proof uses the higher-dimensional analogue of the Shafarevich boundedness conjecture for families of canonically polarized varieties,  finiteness results for maps to varieties of general type, and rigidity theorems for Lefschetz pencils of complete intersections.
\end{abstract}

\maketitle

\section{Introduction}  
The aim of this paper is  to prove a Belyi-type characterisation of smooth complete intersections of general type over $\CC$ which can be defined over $\Qbar$. In other words, we generalize Belyi's theorem for curves to certain higher-dimensional varieties. To motivate this, let us briefly discuss applications of Belyi's theorem for curves.

Belyi's theorem for curves was famously used by Grothendieck to show that the action of the absolute Galois group of $\mathbb Q$ on the set of Galois dessins is faithful \cite[Theorem 4.7.7]{Szamuely}. This result started a flurry of activity on dessins d'enfants \cite{Schneps}. Subsequently, the action of the Galois group of $\mathbb Q$ on the set of connected components of the coarse moduli space of surfaces of general type was proven to be faithful by Bauer-Catanese-Grunwald \cite{BCG}; see also the work of Easton-Vakil \cite{Vakil} and Gonz\'alez-Diez--Torres-Teigell \cite{Torres}.  

To state our main theorem, let $\QQ\to \Qbar$ be the algebraic closure of $\QQ$ in $\CC$. Let $X$ be a smooth projective connected curve over $\CC$.  If $X$ can be defined over $\Qbar$, then Belyi proved that there exists a morphism $X\to \PP^1_{\CC}$ ramified over at most three points \cite{Belyi}. Conversely, by classical results of Weil-Grothendieck \cite{SGA1, Kock, Weil}, if there exists a non-constant morphism $X\to \PP^1_{\CC}$ ramified over precisely three points, then $X$ can be defined over $\Qbar$. In other words, $X$ can be defined over $\Qbar$ if and only if $X$ admits a rational function with at most three critical points.

The main result  of this paper states that a smooth complete intersection $X$ of general type over $\CC$ can be defined over $\Qbar$ if and only if $X$ admits a Lefschetz function with only three critical points (see  Definition \ref{defn: lefschetz}). Here by a Lefschetz function on $X$ is meant a rational function $X\dashrightarrow \PP^1_{\CC}$ on $X$ which factors via a Lefschetz pencil $X\dashrightarrow  \PP^1_{\CC}$ and a rational function $\PP^1_{\CC}\to\PP^1_{\CC}$ on the projective line.

\begin{thm}\label{thm: intro}
Let $X$ be a smooth  complete intersection of general type over $\CC$. Then the following are equivalent:
\begin{enumerate}
\item the variety $X$ can be defined over $\Qbar$, and
\item there exists a Lefschetz function $X\dashrightarrow \PP^1_\CC$ with at most three critical points.
\end{enumerate}
\end{thm}

For two-dimensional varieties, our theorem follows from a result of  Gonz\'alez-Diez \cite[Theorem 1]{Gon1}. Our restriction to complete intersections is necessary to prove the rigidity of Lefschetz functions (see Section \ref{section:rigidity}). Moreover, our restriction to varieties of general type allows us to use results from the MMP specific to varieties of general type (see Theorem \ref{thm: mmp}). It also allows us to invoke boundedness results for families of canonically polarized varieties (Theorem \ref{thm: kovacslieblich}).  

  To prove Theorem \ref{thm: intro} we follow closely the strategy of Gonz\'alez-Diez.  Indeed, the proof of Gonz\'alez-Diez combines 1) well-known finiteness results for minimal models of surfaces with 2) boundedness and 3) rigidity theorems due to Arakelov \cite{ArakelovShaf} and Parshin \cite{Parshin3, Parshin} for families of curves.  Although the boundedness of families of canonically polarized varieties has been proven by Kov\'acs-Lieblich \cite{KovacsLieblich}, generalizing the strategy of Gonz\'alez-Diez to higher-dimensional varieties of general type poses several technical difficulties which we briefly discuss now.

For instance, the rigidity results for families of curves of genus at least two invoked by Gonz\'alez-Diez do not hold in general, as they fail  for several families of higher-dimensional varieties (see \cite[Example 3.1]{KovacsStrong}). To remedy the situation, our novel contribution is a rigidity theorem for certain Lefschetz pencils (Theorem \ref{thm: rigidity of Lefschetz}). In the proof of this result we    use the theory of Higgs bundles \cite{Simpson2, Simpson1}. 

In dimension at least five, the termination of flips is not known currently and complicates proving the appropriate finiteness results for minimal models of varieties of general type (as used by Gonz\'alez-Diez for two-dimensional varieties). To circumvent these difficulties we instead use a classical theorem of Kobayashi-Ochiai \cite{KO} to prove a higher-dimensional analogue (Theorem \ref{thm: mmp}) of the finiteness results used by Gonz\'alez-Diez.

\subsection*{Acknowledgements} We would like to gratefully thank Chenyang Xu for his help in writing Section \ref{section: mmp}, and especially for providing us with a proof of Theorem \ref{thm: mmp}. We thank Chris Peters for a valuable discussion. We also thank  S\'andor Kov\'acs, Duco van Straten and Kang Zuo for inspiring discussions. We thank Jean-Beno\^it Bost,  Yohan Brunebare, Bas Edixhoven, David Holmes,  Robin de Jong, Daniel Loughran and John Voight for helpful discussions on the Shafarevich conjecture and Belyi's theorem.  We gratefully acknowledge the support of   SFB/Transregio 45. 

\subsection*{Conventions} 
The base field is an algebraically closed field $k$ of characteristic zero. 

A variety over a field $k$ is an integral quasi-projective scheme over $k$. A curve is a one-dimensional variety. A smooth projective variety $X$ is of \textit{general type} if  the Kodaira dimension of (the canonical bundle of) $X$ equals $\dim X$. A projective variety $X$ is of \textit{general type} if, for some (hence any) desingularization $X^\prime\to X$ with $X^\prime$  projective, the smooth projective variety $X^\prime$ is of general type. 

The group of automorphisms of the field of complex numbers $\CC$ is denoted by $\mathrm{Aut}(\CC)$.

\section{Finiteness results}
 In this section we first establish a finiteness result for varieties of general type dominated birationally by a fixed variety; see Theorem \ref{thm: mmp}.  To prove this finiteness result, we will use basic results in the minimal model program  such as Mori's cone theorem.

Next we state a finiteness result (due to Kov\'acs-Lieblich) for the number of deformation types of a family of canonically polarized varieties over a fixed base space; see Theorem \ref{thm: kovacslieblich}. This theorem has a long history going back to Shafarevich and we refer to \cite{KovacsSurvey} for a discussion of this.  

 Finally, following the strategy of Viehweg-Zuo \cite{VZ4} and Yi Zhang \cite{YiZhang}, we prove a rigidity theorem for certain Lefschetz fibrations; see Theorem \ref{thm: rigidity of Lefschetz}.

\subsection{An application of the minimal model program}\label{section: mmp}

Our first finiteness result (Theorem \ref{thm: mmp}) concerns varieties of general type dominated birationally by a fixed variety.  
\begin{lem}\label{lem: mmp}
Let $X$ be a smooth projective variety. Let $H$ be an ample line bundle on $X$ and $c$ a real number. Then, there are only finitely many numerical equivalence classes $[L]\in \mathrm{N}_1(X)$ of big and nef line bundles $L$ on $X$ such that \[L^{\dim X-1}\cdot H\leq c.\]
\end{lem}

\begin{proof} Write $n=\dim X$. Replacing $H$ by a positive multiple if necessary, we may and do assume that $H-K_X$ is ample. Let $L$ be a big and nef line bundle on $X$ such that $L^{n -1} \cdot H \leq c.$  As the line bundle $H-K_X + L$ is ample,   \begin{eqnarray}\label{eqn} 
(H+L)^{n-1}\cdot K_X &<& (H+L)^n = \sum_{i=0}^n \binom{n}{i} H^i \cdot L^{n-i}.
\end{eqnarray}  By the Hodge index theorem \cite[Proposition 2.5.1]{BelSom}, for all $i = 1,\ldots,n$, the inequality \[ H^i\cdot L^{n-i} \leq (H\cdot L^{n-1})^i \] holds. As $L^n \leq H \cdot L^{n-1}$, by (\ref{eqn}), we conclude that 
\begin{eqnarray*}
	 (H+L)^{n-1} \cdot K_X &\leq & H\cdot L^{n-1} + \sum_{i=1}^n \binom{n}{i} (H\cdot L^{n-1})^i \\ &\leq& c+ \sum_{i=1}^n \binom{n}{i} c^i \\ &=& c\left(\frac{c^n+c -2 }{c-1} \right).
	 \end{eqnarray*}
 
Matsusaka's big theorem   \cite{KollarMat} 
and the above inequality imply   that there exists a positive integer $N  = N(c,n)$ depending only on $c$ and $n$ such that $N(H+L)$ is very ample. Thus, there is an integer $d$ (depending only on $c$, $H$ and $n$) such that $N(H+L)$ defines a closed immersion $f:X\to \PP^d$ with $f^\ast \mathcal O_{\PP^d}(1) = N(H+L)$. Therefore, the set of numerical equivalence classes of $N(H+L)$, as $L$ runs through all big and nef line bundles on $X$ such that $L^{n-1}\cdot H\leq c$, is finite. This concludes the proof of the lemma.
\end{proof}

As an application of Lemma \ref{lem: mmp} we now prove a finiteness statement concerning varieties of general type dominated birationally by a fixed variety.

\begin{thm}\label{thm: mmp} 
Let $Y$ be a smooth projective variety. Let $n$ be a positive integer,  $D$ an ample line bundle on $\mathbb P^n_k$ and $d$ an integer.  Then the set of isomorphism classes of smooth projective varieties of general type $X$  such that there exists a (proper surjective) birational morphism $Y\to X$ and a closed immersion $f:X\to \mathbb P^n_{k}$ of $D$-degree $d$ is finite. 
\end{thm}

\begin{proof} To prove the theorem, we may and do assume that $Y$ is of general type. Moreover, to ease the notation, we let $\mathcal O(1) := \mathcal O_{\mathbb P^n_k}(1)$.

Since $f^\ast \mathcal O(1)$ is ample on $X$, by Mori's cone theorem (\cite[Theorem 6.1]{Debarre} or \cite[Theorem 3.7]{KollarMori}), the line bundle  $K_X + (\dim Y + 2)f^\ast \mathcal O(1)$ is ample (as $\dim Y+2 = \dim X+2 > \dim X +1$). Write $L = (\dim Y+2) f^\ast \mathcal O(1)$.

Fix an effective ample divisor $H$ on $Y$. Then, by Lemma \ref{lem: mmp}, to prove the theorem, it suffices to show that  $(\phi^\ast L)^{\dim Y  -1} \cdot H$ is bounded by a real number depending only on $Y$,  $H$, $n$, $D$ and $d$. 

To do so, note that since $Y$ is of general type, by \cite[Proposition 2.2.7]{Lazarsfeld}, there exists a positive integer $m$ with  \[m K_Y \geq H.\] 
Here $m$ depends only on $Y$ and $H$. It is straightforward to see that \[(\phi^\ast L)^{\dim Y -1} \cdot H \leq m (\phi^\ast L)^{\dim Y - 1} \cdot K_Y  = m L^{\dim Y-1} \cdot K_X.\] Write $N = \dim Y+ 2 = \dim X+2$. By assumption, the real number   
\begin{align*}
m L^{\dim Y-1} \cdot K_X & = m N^{\dim Y -1} (f^\ast \mathcal O(1))^{\dim Y-1} \cdot K_X  \\ 
&= mN^{\dim Y - 1} \mathcal O(1)^{\dim Y-1}\cdot (f_\ast K_X)
\end{align*} is (bounded by) a real number depending only on $Y$, $H$, $n$, $D$ and $d$. This concludes the proof.
\end{proof}

\begin{opm}\label{remark}
It seems reasonable to suspect that, for $X$  a smooth projective variety, the set of $k$-isomorphism classes of smooth projective varieties of general type $Y$ over $k$ such that there exists a (proper surjective) birational morphism $Y\to X$ is finite. In other words, the conclusion of Theorem \ref{thm: mmp} should hold without the assumption that the $X$ form a bounded family. Indeed, if $\dim Y \leq 3$, this is \cite[Theorem 4]{Tsai}. If $\dim Y =4$,  using the current state-of-the-art in the MMP, the proof of Tsai (given in \textit{loc. cit.}) also works by the MMP for varieties of general type (Birkar-Cascini-Hacon-McKernan  \cite[Theorem B, p. 11]{BCHM}) and the termination of flips in dimension four (Hacon-McKernan-Xu \cite[Corollary 1.2]{HMX}). Thus, in general,  ``termination of flips''  implies that the conclusion of Theorem 2.2 holds without the assumption that the $X$ form a bounded family.
\end{opm}

\subsection{Shafarevich's boundedness conjecture}\label{section: shaf bound}
A smooth projective variety over a field $k$ is \textit{canonically polarized} if  its canonical bundle  $\omega_{X/k}$  is ample.  If $X$ is canonically polarized, the \textit{Hilbert polynomial of $X$} is the Hilbert polynomial of $\omega_{X/K}$. Note that canonically polarized varieties are of general type.

Let $B$ be a  variety over an algebraically closed field $k$ of characteristic zero. A smooth projective morphism of $k$-varieties $f:Y\to B$ is a \textit{family over $B$} if it has (geometrically) connected fibres. If $f_1:Y_1\to B$ and $f_2:Y_2\to B$ are families over $B$, a \textit{morphism from $f_1:Y_1\to B$ to $f_1:Y_2\to B$} is a morphism of $k$-varieties $\phi:Y_1\to Y_2$ such that $f_1 = f_2 \circ \phi$.

If the fibres of a family $f:Y\to B$ over $B$ are canonically polarized we say that  $f$ is a \textit{family of canonically polarized varieties}. Moreover, the \textit{Hilbert polynomial of $f:Y\to B$} is the Hilbert polynomial of the canonically polarized variety $Y_b$, where $b$ is some (hence any) closed point of $B$.

 A family $f:Y\to B$ over $B$ is \textit{trivial} if there is a (smooth projective) variety $F$ over $\CC$ such that $Y$ is isomorphic to $F\times_{k} B$ over $B$. A family $f:Y\to B$ is \textit{isotrivial} if all of its closed fibres are isomorphic. If $f:Y\to B$ is a family of canonically polarized varieties over a smooth curve $B$, then $f$ is isotrivial if and only if there exists a finite \'etale morphism $C\to B$ such that $Y\times_B C\to C$ is a trivial family over $C$ \cite[Lemma 7.3]{KebekusKovacs3}. 

Let $h$ be a polynomial. If $f:Y\to B$ is a family of canonically polarized varieties with Hilbert polynomial $h$, a \textit{deformation of $Y\to B$ (as a family over $B$)} is a triple $(T,t_0, \psi)$, where $T$ is a connected variety over $k$, $t_0$ is a closed point of $T$ and $\psi:\mathcal Y\to B\times_k T$ is a family of canonically polarized varieties with Hilbert polynomial $h$ over $B\times_k T$ such that $f:Y\to B$ is isomorphic to the family $\mathcal Y_{t_0}\to B\times_k \{t_0\} \cong B$ over $B$. If $f_1:Y_1\to B$ and $f_2:Y_2\to B$ are families of canonically polarized varieties with Hilbert polynomial $h$, then $f_1$ and $f_2$ are \textit{deformation equivalent (as families of canonically polarized varieties with Hilbert polynomial $h$ over $B$)} if there is a deformation $(T,t_0, \psi)$ of $f_1:Y_1\to B$ and a closed point $t\in T$ such that   $f_2:Y_2\to B$ is isomorphic to the family $\mathcal Y_{t}\to B\times_k \{t\} \cong B$. This defines an equivalence relation on the set of families $Y\to B$ of canonically polarized varieties with Hilbert polynomial $h$. An equivalence class (with respect to this equivalence relation) is called a \textit{deformation type}. With these definitions,  the work of Viehweg-Zuo \cite{VZ1, VZ2, VZ5} on Arakelov inequalities (or ``weak boundedness'') was used by Kov\'acs-Lieblich \cite{KovacsLieblich} to  prove the following analogue of the Shafarevich boundedness conjecture for families of canonically polarized varieties:

\begin{thm}\label{thm: kovacslieblich} \cite[Theorem 1.6]{KovacsLieblich}
Let $B$ be a smooth variety over $\CC$ and $h$ a polynomial. Then the set of  deformation types of non-isotrivial families of canonically polarized varieties $Y\to B$  over $B$ with Hilbert polynomial $h$ is finite.\qed
\end{thm}

Let $f:Y\to B$ be a family of canonically polarized varieties.  We say that $f: Y\to B$ is \emph{rigid} if any deformation $(T,t_0,\psi)$ of $f$ has the property that, for all $t$ in $T$, the family $\psi_t:\mathcal Y \to B\times \{t\} \cong B$ is isomorphic to $f:Y\to B$. In other words, any deformation of the family $Y\to B$ is ``trivial''.  If $\dim X - \dim B = 1$, then $f$ is rigid \cite{ArakelovShaf}. If $\dim X - \dim B>1$, it is easy to construct non-rigid families; see \cite[Example 3.1]{KovacsStrong} or \cite{Patakfalvi}.
In \cite[Theorem 0.3]{VZ4}, it is shown that for all $n\geq 2$ and $d\geq n+2$, there exists a smooth variety $B$ and a non-rigid family of canonically polarized varieties $X\to B$ whose fibers are degree $d$ hypersurfaces in $\mathbb P^{n+1}$.  

As we have explained above,  there are non-rigid  families of hypersurfaces (and complete intersections). We will show in the next section how to prove the rigidity of a Lefschetz fibration of complete intersections, under suitable assumptions.

\subsection{Rigidity of Lefschetz fibrations}\label{section:rigidity}

In \cite{VZ5} Viehweg and Zuo prove the rigidity of certain families of hypersurfaces by studying the effect of a deformation on the associated variation of Hodge structures. Their strategy was  applied in the context of Lefschetz pencils on Calabi-Yau varieties in \cite[Theorem 4.1]{YiZhang}. In this section we follow the strategy of Viehweg-Zuo to prove the rigidity of  Lefschetz pencils of smooth complete intersections (Theorem \ref{thm: rigidity of Lefschetz}).
 


Any smooth projective variety $X$ over $k$ admits a Lefschetz pencil \cite[Expos\'e XVII, Th\'eor\`eme 2.5]{SGA7II}. Note that a Lefschetz pencil on $X$ induces a rational function $X\dashrightarrow \PP^1_k$. A flat proper morphism $f:Y\to \PP^1_k$ is a \emph{Lefschetz fibration} if there exist a   Lefschetz pencil $X\dashrightarrow \mathbb P^1_k$ and a birational morphism $Y\to X$  such that $Y\to \mathbb P^1_k$ factors via  $X\dashrightarrow \PP^1_k$ and $Y\to X$ is the blow-up of $X$ along the axes of the Lefschetz pencil. We will consider only Lefschetz pencils of complete intersections. Let us be more precise.


 A \emph{type} is a collection of integers $T=(d_1,\ldots,d_c;n),$
    which satisfy the inequalities
    $n \geq1,   c\geq1,$, and $ 2\leq d_1\leq \ldots \leq d_c.$
    A \textit{complete intersection of type $T$ over $k$} is a closed subscheme of codimension $c$ in $\PP^{n+c}_k$ 
    given as the zero locus of $c$ homogeneous equations of degrees $d_1,\ldots,d_c$.

 A family of complete intersections is rigid if any deformation of the family (as a family of complete intersections) is trivial.   We now use the theory of Higgs bundles \cite{Simpson2, Simpson1} and prove that any Lefschetz pencil of smooth complete intersections is rigid, under suitable assumptions. 

\begin{thm}\label{thm: rigidity of Lefschetz}   Let $T$ be a type such that $T \neq (2;n), (2,2;n), (3;2)$. 
If  $f:Y\to \PP^1_k$ is a Lefschetz fibration of smooth complete intersections of type $T$, then  $f$ is rigid.  
\end{thm}
\begin{proof}
 To prove the theorem, we reason by contradiction and assume (without loss of generality) that $k=\CC$. 
Let $B_0\subset \PP^1_k$ be the largest open subscheme such that $f$ is smooth over $B_0$. Write $f^0:Y_0\to B_0$ for the restriction of $f$ to $B_0$.  Suppose that $f$ is non-rigid. Then $f^0$ is non-rigid as a family of smooth complete intersections over $B_0$. 
In particular, the Higgs bundle $(E,\theta)$ associated to the polarized variation of Hodge structures $\mathrm{End}(R^n_{\mathrm{prim}} f^0_{\ast} \mathbb C)^{-1,1}$ on $B_0$ has a  non-zero flat section $\sigma$. (Indeed, by Flenner's infinitesimal Torelli theorem   for smooth complete intersections of type $T$ (see \cite[Theorem 3.1]{Flenner}),    the differential of the period map associated to the variation of Hodge structures $R^n_{\prim} f^0_{\ast}\mathbb C$ is injective, i.e., for all $b$ in $B_0$, the morphism 
\begin{align}\label{e} \mathrm{H}^1(X_b,T_{X_b}) & \longrightarrow \bigoplus_{p+q = n} \Hom( \mathrm{H}^{p,q}_{\mathrm{prim}}(X_b), \mathrm{H}_{\mathrm{prim}}^{p-1,q+1}(X_b)) \end{align}
is everywhere injective. As $f^0$ admits a non-trivial deformation, the Kodaira-Spencer map induces a non-zero flat section of $R^1 f_{\ast}^0 T_{f_0}$. The existence of the non-zero flat section now follows from the injectivity of (\ref{e}) and  \cite[Corollary 12]{PetersSteenbrink}.)

We now use the existence of a non-zero flat section of the Higgs bundle $(E,\theta)$ to prove the theorem. First, note that the non-zero flat section $\sigma$ of $(E,\theta)$ induces a splitting of Higgs bundles \[ (E,\theta)  = \ker(\sigma)\oplus \ker(\sigma)^{\perp},\] where $\ker(\sigma)^\perp$ is the orthogonal complement of $\ker(\sigma)$ with respect to the polarization  on $E$.  Let $\mathbb V$ be the sub-variation of Hodge structures of  $R^n_{\mathrm{prim}} f^0_{\ast} \mathbb C$ corresponding to the vanishing cycles.  Consider the sub-Higgs bundle $$E^\prime = \ker(\sigma)\cap \mathbb V$$ of $\mathbb V,$ and note that $E^{n,0}$ is contained in $E^\prime$, so that $E^\prime$ is non-zero.  Moreover, again by the infinitesimal Torelli theorem for smooth complete intersections of type $T$, the vector bundle $E^{n,0}$ is not in $\ker(\sigma)$. Therefore, the splitting $(E,\theta) = \ker \sigma \oplus \ker(\sigma)^{\perp}$ is non-trivial and induces a non-trivial splitting of $\mathbb V$.  This contradicts the absolute irreducibility of $\mathbb V$ \cite[Corollary 10.23]{PetersSteenbrink}.
\end{proof}

\begin{opm} 
In   \cite{KovacsStrong, Peng, VZd} the rigidity of a Lefschetz pencil with maximal Yukawa coupling is proven. Moreover, in \cite[Example 3 after Theorem 3.8, p. 123]{Peters} Peters proved the rigidity of families of hypersurfaces with surjective Kodaira-Spencer map. Note that Theorem \ref{thm: rigidity of Lefschetz} does not   require any hypothesis on the Yukawa coupling or on the Kodaira-Spencer map. 
\end{opm}

We will say that a type $T= (d_1,\ldots,d_c;n)$ is  of general type if $d_1+\ldots+ d_c \geq n+c+2$.
Note that the adjunction formula implies that a smooth complete intersection of type $T$ is of general type if and only if
$T$ is of general type.

\begin{cor}\label{cor: hypersurfaces defn}
  Let $T$ be a type of general type and let $B_0\subset B=\PP^1_k$ be a dense open.  Then the set of $B$-isomorphism classes of Lefschetz fibrations of complete intersections of type $T$ over $B$ which are smooth over $B_0$ is finite.
\end{cor}
\begin{proof}
As $T$ is of general type, any smooth complete intersection of type $T$ is  canonically polarized.  Thus,  the corollary follows from the boundedness (Theorem \ref{thm: kovacslieblich}) of smooth families of canonically polarized varieties over $B_0$, and rigidity  (Theorem \ref{thm: rigidity of Lefschetz}) of Lefschetz fibrations over $B$ of complete intersections of type $T$.
\end{proof}

\begin{cor}\label{cor: lefschetz pencils} Let $T$ be a type and let $B_0\subset \PP^1_k$ be a dense open. 
The set of  varieties of general type $X$ such that $X$ admits a Lefschetz pencil $X\dashrightarrow \PP^1_k$ of type $T$ which is defined and smooth over $B_0$ is finite.
\end{cor}
\begin{proof} This follows from Corollary \ref{cor: hypersurfaces defn} and  Theorem \ref{thm: mmp}.
\end{proof}

\section{Varieties defined over $\Qbar$}
A morphism of $\CC$-varieties $f:X\to Y$ \textit{can be defined over $\Qbar$} if and only if there is a morphism $f_0:X_0\to Y_0$ of $\Qbar$-varieties, an isomorphism $X\isomto X_0\otimes_{\Qbar} \CC$ and an isomorphism $Y\isomto Y_0\otimes_{\Qbar} \CC$ such that the following diagram \[\xymatrix{ X \ar[d]_f \ar[r]^-{\sim} &  X_0\otimes_{\Qbar}\CC \ar[d]^{f_0\otimes_{\Qbar} \CC} \\  Y \ar[r]^-{\sim}& Y_0\otimes_{\Qbar}\CC } \] is commutative. A variety $X$  \textit{can be defined over $\Qbar$} if there exists a variety $X_0$ over $\Qbar$ and an isomorphism $X \isomto X_0\otimes_{\Qbar} \CC$. (Note that $X$ can be defined over $\Qbar$ if and only if the identity morphism $\mathrm{id}_X:X\to X$ can be defined over $\Qbar$.)

The aim of this section is to state a criterion for a morphism of projective varieties over $\CC$ to be defined over $\Qbar$. We will follow \cite{ Gon2, Gon3, Gon1}. Let us recall how $\mathrm{Aut}(\CC)$ acts on the set of isomorphism classes of projective varieties over $\CC$.

For $\sigma$ in $\mathrm{Aut}(\CC)$ and $f$ in $\CC[x_0,\ldots,x_n]$, let $f^\sigma$ be the polynomial obtained by acting with $\sigma$ on its coefficients. This defines an action \[ \mathrm{Aut}(\CC) \times \CC[x_0,\ldots, x_n] \to \CC[x_0,\ldots,x_n], \quad (\sigma, f) \to f^\sigma\] of $\mathrm{Aut}(\CC)$ on the ring $\CC[x_0,\ldots,x_n]$. If $I$ is an ideal of $\CC[x_0,\ldots,x_n]$ and $\sigma$ is an automorphism of $\CC$, we define the ideal  $I^\sigma$ to be the image of $I$ under $\sigma$. In other words, $I^\sigma$ consists of the elements $f^\sigma$, where $f\in I$. This induces an action of $\mathrm{Aut}(\CC)$ acts on the set of isomorphism classes of projective varieties over $\CC$.   For $X$ a  projective variety, we let $X^\sigma$ denote (a representative of the isomorphism class of) its conjugate under the action of $\sigma \in \mathrm{Aut}(\CC)$. 

\begin{lem}\label{lem: varieties over nf} \cite[p. 60-61]{Gon1} Let $X$ be a projective variety over $\CC$. Then $X$ can be defined over $\Qbar$ if and only if the set of isomorphism classes of conjugates $X^\sigma$ of $X$ (as $\sigma$ runs over all elements of $\mathrm{Aut}(\CC)$) is finite. \qed
\end{lem} 

For a morphism $f:X\to Y$ of $\CC$-varieties and $\sigma$ in $\mathrm{Aut}(\CC)$, we let $f^\sigma$ denote (a representative of the isomorphism class of) its conjugate.

\begin{lem}\label{lem: morphisms over nf} \cite[p. 60-61]{Gon1} Let $f:X\to Y$ be a morphism of projective varieties over $\CC$. Then $f$ can be defined over $\Qbar$ if and only if the set of isomorphism classes of conjugates $f^\sigma$ of $f$ is finite. \qed
\end{lem}

We now combine our analogue (Theorem \ref{thm: mmp}) of Tsai's finiteness result (see Remark \ref{remark}) with the criteria for a variety to be defined over $\Qbar$, and show that if $X$ is a variety of general type over $\CC$ that is dominated by a variety which can be defined over $\Qbar$, then $X$ can be defined over $\Qbar$.

\begin{lem}\label{lem: dominant}
Let $X$ be a projective variety of general type over $\CC$. Let $f:Y\to X$ be a birational morphism, where $Y$ is a projective variety that can be defined over $\Qbar$. Then  $X$ and $f$ can be defined over $\Qbar$.
\end{lem}
\begin{proof} (We follow the proof of \cite[Proposition 3.2]{Gon3}.) By resolution of singularities in characteristic zero, there exists a surjective morphism $f:Y^\prime\to X$, where $Y^\prime$ is a smooth projective variety that can be defined over $\Qbar$. (Indeed, let $Y^\prime\to Y$ be a resolution of singularities with a model over $\Qbar$.) To prove the lemma, it suffices to show that the set of (isomorphism classes of) conjugates $f^\sigma$ (with $\sigma \in \mathrm{Aut}(\CC)$) is finite. To do so, note that by Lemma \ref{lem: varieties over nf}, the set of isomorphism classes of conjugates $Y^\sigma$ is finite. Moreover, by the Kobayashi-Ochiai finiteness theorem \cite{KO}, for all $\sigma\in \mathrm{Aut}(\CC)$, the set of surjective morphisms from $Y^\sigma$ to  a fixed variety of general type is finite. Also, since $X$ is projective, there exists a closed embedding $X\subset \PP^n_\CC$ of some degree $d$. Note that any conjugate $X^\sigma$ of $X$ admits a closed embedding $X^\sigma \subset \PP^n_\CC$ of the same degree $d$.

Therefore, to conclude the proof, it suffices to show that the set of $\CC$-isomorphism classes of smooth projective varieties of general type $X$ over $\CC$ such that there exists a birational morphism $Y\to X^\prime$ and a closed immersion $X^\prime\subset \PP^n_\CC$ of degree $d$ is finite. The latter clearly follows from Theorem \ref{thm: mmp} .
\end{proof}

\begin{opm}
The converse to Lemma \ref{lem: dominant} is false. For $n\geq 2$, the (total space of the) blow-up of $\PP^n_{\CC}$ in a set of points $B\subset \PP^n(\CC)$ of cardinality $>n+2$ with ``transcendental'' coordinates can not be defined over $\Qbar$.
\end{opm}

\begin{lem}\label{cor: lefschetz fibn over qbar} Let $Y$ be a smooth complete intersection of general type over $\CC$ and let $f:Y\to \PP^1_{\CC}$ be a Lefschetz fibration whose critical points lie in $\PP^1(\Qbar)$. Then the variety $Y$ and the morphism $f$ can be defined over $\Qbar$.
\end{lem}
\begin{proof} By Lemma \ref{lem: varieties over nf}, to prove the corollary, it suffices to show that the set of conjugates $f^\sigma$ of $f$ under the action of $\mathrm{Aut}(\CC)$ is finite.  Let $B\subset \PP^1(\Qbar)$ be the set of critical points of the Lefschetz fibration $f:Y\to \PP^1_{\CC}$. Note that, as $Y$ is a complete intersection, the fibres of $f$ are complete intersections of some type, say $T$. Note that $T$ is of general type. Since the set of conjugates $B^\sigma$ of $B$ is finite and $f^\sigma$ is a Lefschetz fibration of complete intersections of (the same) type $T$, it suffices to show that the set of Lefschetz fibrations of fixed type which are smooth over a fixed open in $\PP^1_\CC$ is finite. The latter clearly follows from Corollary \ref{cor: lefschetz pencils}.
\end{proof}

\begin{opm}
Let $X$ be a   complete intersection over $\CC$. It follows from \cite[Proposition 2.1.11]{BenoistThesis} that $X$ can be defined over $\Qbar$ as a variety if and only if $X$ can be defined over $\Qbar$ as a complete intersection.
\end{opm}

\section{Lefschetz functions: proof of Theorem \ref{thm: intro}}
Let $X$ be a smooth projective variety over $\CC$. If $X$ can be defined over $\Qbar$, there exists a rational function $X\dashrightarrow \PP^1_\CC$ on $X$ whose critical points lie in $\mathbb P^1_{\Qbar}(\Qbar)$. The converse to this statement is false. Indeed, if $C$ and $D$ are smooth projective curves of genus at least two such that $C$ can be defined over $\Qbar$ and $D$ can't be defined over $\Qbar$, then $X = C\times D^{n-1}$ is an $n$-dimensional variety of general type that can't be defined over $\Qbar$. Nevertheless, there is a rational function on $X$ with at most three critical points. Explicitly: let $\pi:C\to \PP^1_{\CC}$ be a Belyi morphism and $X\to C\to \PP^1_{\CC}$ the projection followed by $\pi$. This explains why we restrict to a special class of rational functions in Theorem \ref{thm: intro}.

\begin{defn} \label{defn: lefschetz}
A rational function $p:X\dashrightarrow \PP^1_\CC$ is a \textit{Lefschetz function (on $X$)} if there exist
 a Lefschetz pencil $f:X\dashrightarrow  \PP^1_{\CC}$ and a rational function $g:\PP^1_{\CC}\to\PP^1_{\CC}$ such that $p = g\circ f$.
\end{defn}
It is clear that Theorem \ref{thm: intro} follows from the following more general statement.  

\begin{thm}\label{thm: main theorem} Let $X$ be a smooth complete intersection of general type over $\CC$. Then the following are equivalent:
\begin{enumerate}
\item the variety $X$ can be defined over $\Qbar$;
\item there exists a Lefschetz pencil $X\dashrightarrow \PP^1_{\CC}$ whose critical points lie in $\PP^1(\Qbar)$;
\item there exists a Lefschetz function $X\dashrightarrow \PP^1_{\CC}$ whose critical points lie in $\PP^1(\Qbar)$;
\item there exists a Lefschetz function $X\dashrightarrow \PP^1_\CC$ with at most three critical points;
\item there exist  a birational morphism $Y\to X$ and a Lefschetz fibration $Y\to B$  whose critical points lie in $\PP^1(\Qbar)$.
\end{enumerate}
\end{thm}
\begin{proof} The implication $(1)\implies (2)$ follows from the existence of a Lefschetz pencil on a projective variety over $\Qbar$.

As any Lefschetz pencil on $X$ is a Lefschetz function, the implication $(2) \implies (3)$ follows.

Note that $(3) \implies (4)$ follows from Belyi's  algorithm: for all finite subsets $B$ of $\PP^1(\Qbar)$, there exists a finite morphism $R:\PP^1_{\Qbar}\to \PP^1_{\Qbar}$ such that $R$ ramifies only over $\{0,1,\infty\}$ and $R(B)\subset \{0,1,\infty\}$. 

The implication $(4) \implies (5)$ follows from the definition of a Lefschetz function (Definition \ref{defn: lefschetz}). 

Now,  to prove the theorem, it suffices to prove that $(5)\implies (1)$.  To do so, let $Y\to X$ be a birational morphism and let $f:Y\to \PP^1_{\CC}$ be a Lefschetz fibration such that its set $B$ of critical points lies in $\PP^1(\Qbar)$.  Then, by Lemma  \ref{cor: lefschetz fibn over qbar} the variety $Y$ can be defined over $\Qbar$. Finally, by Lemma  \ref{lem: dominant},  we conclude that $X$ can be defined over $\Qbar$.
\end{proof}

\begin{opm}
As there are only finitely many smooth complete intersections $X$ of general type over $\CC$ such that there exists a Lefschetz pencil $X\dashrightarrow \PP^1_{\CC}$ with at most three critical points (Corollary \ref{cor: lefschetz pencils}), the analogous statement of Theorem \ref{thm: main theorem} in which we replace ``Lefschetz function'' by ``Lefschetz pencil'' is false.
\end{opm}

\bibliography{refsbel}{}
\bibliographystyle{plain}

\end{document}